\documentclass[a4paper]{amsart}
\usepackage[utf8]{inputenc}
\usepackage[T2A]{fontenc}
\usepackage[english]{babel}

\usepackage{amssymb}
\usepackage{amsthm}

\newtheorem{lemma}{Lemma}
\newtheorem{theorem}{Theorem}

\def\hm#1{#1\nobreak\discretionary{}{\hbox{\ensuremath{#1}}}{}}

\def\Z{{\mathbb Z}}

\def\CC{{\mathbb C}}

\def\lcm{\mathop{\mathrm{lcm}}}

\def\lae{\lambda^{(e)}}
\def\phie{\phi^{(e)}}

\def\eps{\varepsilon}

\def\res{\mathop{\mathrm{res}}}
\def\llog{\mathop{\mathrm{llog}}}

\def\le{\leqslant}
\def\ge{\geqslant}

\def\p{\phantom0}

\def\sp{\!\!\!\!\!}

\usepackage[unicode]{hyperref}
\usepackage{breakurl}

\begin{document}

\title{Exponential Carmichael function}
\author{Andrew V. Lelechenko}
\address{I.~I.~Mechnikov Odessa National University}
\email{1@dxdy.ru}

\keywords{Exponential divisors, Carmichael function, moments of Riemann zeta-function}
\subjclass[2010]{
11A25  
11M06, 
11N37, 
11N56
}

\begin{abstract}
Consider exponential Carmichael function $\lae$ such that $\lae$ is multiplicative and $\lae(p^a) \hm= \lambda(a)$, where $\lambda$ is usual Carmichael function. We discuss the value of $\sum \lae(n)$, where $n$ runs over certain subsets of $[1,x]$, and provide bounds on the error term, using analytic methods and especially estimates of $\int_1^T \bigl| \zeta(\sigma+it) \bigr|^m dt$.
\end{abstract}

\maketitle

\section{Introduction}

Consider an operator $E$ over arithmetic functions such that
for every $f$ the function $Ef$ is multiplicative and
$$ (Ef)(p^a) = f(a), \qquad p \text{~is prime}. $$

For various functions $f$
(such as the divisor function,
the sum-of-divisor function,
Möbius function, the totient function and so on)
the behaviour of $Ef$
was studied by many authors,
starting from Subbarao~\cite{subbarao1972}.
The bibliography can be found in~\cite{bibliography}.

The notation for $Ef$, established by previous authors, is $f^{(e)}$.

\medskip

Carmichael function $\lambda$ is an arithmetic function such that
$$
\lambda(p^a) = \begin{cases}
\phi(p^a), & p>2 \text{~or~} a=1,2, \\
\phi(p^a)/2, & p=2 \text{~and~} a>2,
\end{cases}
$$
and if $n=p_1^{a_1} \cdots p_m^{a_m}$ is a canonical representation, then
$$
\lambda(n) = \lcm\bigl( \lambda(p_1^{a_1}), \ldots, \lambda(p_m^{a_m}) \bigr).
$$

This function was introduced at the beginning of the XX century in \cite{carmichael1909}, but intense studies started only in 1990-th, e. g. \cite{erdos1991}. Carmichael function finds applications in cryptography, e. g. \cite{friedlander1999}.

\medskip
Consider also the family of multiplicative functions
$$
\delta_r(p^a) = \begin{cases}
0, & a<r, \\
1, & a\ge r,
\end{cases}
\qquad r \text{~is integer.}
$$

Function $\delta_2$ is a characteristic function of the set of square-full numbers, $\delta_3$ --- of cube-full numbers and so on. Of course, $\delta_1 \equiv 1$.

Denote $\lae_r$ for the product  of $\delta_r$ and $\lae$:
$$
\lae_r(n) = \delta_r(n) \lae(n).
$$

The aim of our paper is to study asymptotic properties of
$\lae \equiv \lae_1$, $\lae_2$, $\lae_3$ and $\lae_4$.

Note that all proofs below remains valid for $\phie_r(n) = \delta_r(n) \phie(n)$ instead of~$\lae_r(n)$ for $r=1,2,3,4$.

\section{Notations}

Letter $p$ with or without indexes denotes a prime number.

We write $f\star g$ for Dirichlet convolution
$$ (f \star g)(n) = \sum_{d|n} f(d) g(n/d). $$

Denote
$$
\tau(a_1,\ldots,a_k; n) := \sum_{d_1^{a_1}\cdots d_k^{a_k} = n} 1.
$$

In asymptotic relations we use $\sim$, $\asymp$, Landau symbols $O$ and $o$, Vinogradov symbols $\ll$ and $\gg$ in their usual meanings. All asymptotic relations are given as an argument (usually $x$) tends to the infinity.

Everywhere $\eps>0$ is an arbitrarily small number (not always the same even in one equation).

As usual $\zeta(s)$ is Riemann zeta-function.
Real and imaginary components of the complex $s$ are denoted as $\sigma:=\Re s$ and~$t:=\Im s$, so $s=\sigma+it$.


For a fixed $\sigma\in[1/2,1]$ define
$$
m(\sigma) := \sup\biggl\{
m \biggm|
\int_1^T \bigl| \zeta(\sigma+it) \bigr|^m dt \ll T^{1+\eps}
\biggr\}.
$$
and
$$
\mu(\sigma) := \limsup_{t\to\infty} {\log \bigl|\zeta(\sigma+it)\bigr| \over \log t}.
$$

Below $H_{2005}=(32/205+\eps, 269/410+\eps)$ stands for Huxley's exponent pair from~\cite{huxley2005}.

\section{Preliminary lemmas}

\begin{lemma}\label{l:rational-maximal-order}
Let $F\colon \Z\to\CC$ be a multiplicative function such that  $F(p^a) \hm= f(a)$, where $f(n) \ll n^\beta$ for some $\beta>0$. Then
\begin{equation*}
\limsup_{n\to\infty} {\log F(n)\llog n \over\log n} = \sup_{n\ge1} {\log f(n)\over n}.
\end{equation*}
\end{lemma}

\begin{proof}
See \cite{suryanarayana1975}.
\end{proof}

\begin{lemma}\label{l:log-sum}
Let $f(t)\ge 0$. If
$$ \int_1^T f(t) \, dt \ll g(T), $$
where $g(T) = T^\alpha \log^\beta T$, $\alpha\ge 1$,
then
\begin{equation*}
I(T):= \int_1^T {f(t)\over t} dt \ll
\left\{ \begin{matrix}
\log^{\beta+1} T & \text{if } \alpha=1, \\
T^{\alpha-1} \log^{\beta} T & \text{if } \alpha>1.
\end{matrix} \right.
\end{equation*}
\end{lemma}

\begin{proof}
Let us divide the interval of integration into parts:
$$
I(T)
\le
\sum_{k=0}^{\log_2 T}
\int_{T/2^{k+1}}^{T/2^k} {f(t)\over t} dt
<
\sum_{k=0}^{\log_2 T} {1\over T/2^{k+1}}
\int_1^{T/2^k} f(t) dt
\ll
\sum_{k=0}^{\log_2 T} {g(T/2^{k})\over T/2^{k+1}}.
$$
Now the lemma's statement follows from elementary estimates.
\end{proof}

\begin{lemma}\label{l:order-in-the-critical-strip}
For $\sigma\ge1/2$ and for any exponent pair $(k,l)$ such that $l-k \ge \sigma$  we have
$$
\mu(\sigma) \le {k+l-\sigma\over2} + \eps.
$$
\end{lemma}
\begin{proof}
See \cite[(7.57)]{ivic2003}.
\end{proof}

A well-known application of Lemma~\ref{l:order-in-the-critical-strip} is
\begin{equation}\label{eq:mu-1/2}
\mu(1/2) \le 32/205,
\end{equation}
following from the choice $(k,l) = H_{2005}$. Another (maybe new) application is
\begin{equation}\label{eq:mu-3/5}
\mu(3/5) \le 1409 / 12170,
\end{equation}
following from
$$
(k,l) = \left( {269\over 2434}, {1755\over2434} \right) = ABAH_{2005},
$$
where $A$ and $B$ stands for usual $A$- and $B$-processes \cite[Ch. 2]{kratzel1988}.

\begin{lemma}\label{l:phragmen}
Let~$\eta>0$ be arbitrarily small. Then for growing $|t|\ge3$
\begin{equation}\label{eq:convexity}
\zeta(s) \ll \begin{cases}
|t|^{1/2 - (1-2\mu(1/2))\sigma}, & \sigma\in[0, 1/2],
\\
|t|^{2\mu(1/2)(1-\sigma)} , & \sigma\in[1/2, 1-\eta], \\
|t|^{2\mu(1/2)(1-\sigma)} \log^{2/3} |t| , & \sigma\in[1-\eta, 1], \\
\log^{2/3} |t|, & \sigma\ge1.
\end{cases}
\end{equation}
More exact estimates for $\sigma\in[1/2, 1-\eta]$ are also available, e. g.
\begin{equation}\label{eq:convexity-3/5}
\mu(\sigma) \ll \begin{cases}
10\bigr( \mu(3/5)-\mu(1/2) \bigl) \sigma + \bigl( 6\mu(1/2) - 5\mu(3/5) \bigr), & \sigma\in[1/2, 3/5], \\
5\mu(3/5)(1-\sigma)/2 , & \sigma\in[3/5,1-\eta], \\
\end{cases}
\end{equation}
\end{lemma}

\begin{proof}
Estimates follow from Phragmén---Lindelöf principle, exact and approximate functional equations for $\zeta(s)$ and convexity properties.  See \cite[Ch. 5]{titchmarsh1986} and~\cite[Ch. 7.5]{ivic2003} for details.
\end{proof}

\begin{lemma}\label{l:lae-maximal-order}
For any integer $r$
$$
\max_{n\le x} \lae_r(n) \ll x^\eps.
$$
\end{lemma}

\begin{proof}
Surely $\lae_r(n) \le \lae(n)$. By Lemma~\ref{l:rational-maximal-order} we have
$$
\limsup_{n\to\infty} {\log \lae(n) \log\log n \over \log n}
= \sup_m {\log \lambda(m) \over m} = {\log 4\over 5} =: c,
$$
because $\lambda(m) \le m-1$. It implies
$$
\max_{n\le x} \lae(n) \ll x^{c/\log\log n} \ll x^\eps.
$$
\end{proof}

\begin{lemma}\label{l:lae-series}
Let $L_r(s)$ be the Dirichlet series for $\lae_r$:
$$
L_r(s) := \sum_{n=1}^\infty \lae_r(n) n^{-s}.
$$
Then for $r=1,2,3,4$ we have $L_r(s) = Z_r(s) G_r(s)$, where
\begin{align}
\label{eq:lae-series-1}
Z_1(s) &= \zeta(s) \zeta(3s) \zeta^2(5s), \\
\label{eq:lae-series-2}
Z_2(s) &= \zeta(2s) \zeta^2(3s) \zeta(4s) \zeta^2(5s), \\
\label{eq:lae-series-3}
Z_3(s) &= \zeta^2(3s) \zeta^2(4s) \zeta^4(5s), \\
\label{eq:lae-series-4}
Z_4(s) &= \zeta^2(4s) \zeta^4(5s) \zeta^2(6s) \zeta^6(7s),
\end{align}
Dirichlet series $G_1(s)$, $G_2(s)$, $G_3(s)$ converge absolutely for $\sigma>1/6$ and $G_4(s)$ converges absolutely for $\sigma>1/8$.
\end{lemma}

\begin{proof}
Follows from the identities
\begin{align*}
1+\sum_{a\ge1} \lae(p^a) x^a &=
	1+x+x^2+2x^3+2x^4+4x^5+2x^6+6x^7+O(x^8) \\
	&= {1+O(x^8) \over (1-x)(1-x^3)(1-x^5)^2}, \\
1+\sum_{a\ge2} \lae(p^a) x^a &=
	{1+O(x^6) \over (1-x^2)(1-x^3)^2(1-x^4)(1-x^5)^2}, \\
1+\sum_{a\ge3} \lae(p^a) x^a &=
	{1+O(x^6) \over (1-x^3)^2 (1-x^4)^2 (1-x^5)^4}, \\
1+\sum_{a\ge4} \lae(p^a) x^a &=
	{1+O(x^8) \over (1-x^4)^2 (1-x^5)^4 (1-x^6)^2 (1-x^7)^6}.
\end{align*}
\end{proof}

\begin{lemma}\label{l:kratzel}
Let $\Delta(x)$ be the error term in the well-known asymptotic for\-mu\-la for $\sum_{n\le x} \tau(a_1,a_2,a_3,a_4; n)$, let $A_4 = a_1+a_2+a_3+a_4$  and let $(k,l)$ be any exponent pair. Suppose that the following conditions are satisfied:
\begin{enumerate}
\item $(k+l+2) a_4  < (k+l) a_1 + A_4$.
\item $2(k+l+1) a_1 \le (2k+1) (a_2+a_3)$.
\item[(3.1)] $l a_1 \le k a_2$ and $(k+l+1)a_1 \ge k(a_2+a_3)$
\item[\em or]
\item[(3.2)] $l a_1 \ge k a_2$ and $(l-k)(2k+1)a_3 \le (2l-2l-1)(k+l+1)a_1 + \bigl( 2k(k-l+1) + 1 \bigr) a_2$.
\end{enumerate}
\end{lemma}
\begin{proof}
This is \cite[Th. 3]{kratzel1992} with $p=4$.
\end{proof}

\begin{lemma}\label{l:ivic-moments-on-sigma}
$$
m(\sigma) \ge \begin{cases}
\p\p\p\p4/(3-4\sigma), & \p1/2\p \le \sigma \le  \p5/8,  \\
\p\p\p10/(5-6\sigma), & \p5/8\p \le \sigma \le 35/54, \\
\p\p\p19/(6-6\sigma), & 35/54 \le \sigma \le 41/60, \\
\p2112/(859-948\sigma), & 41/60 \le \sigma \le \p3/4, \\
12408/(4537-4890\sigma), & \p3/4\p \le \sigma \le \p5/6, \\
\p4324/(1031-1044\sigma), & \p5/6\p \le \sigma \le \p7/8, \\
\p\p\p98/(31-32\sigma), & \p7/8\p \le \sigma \le 0.91591\ldots, \\
(24\sigma-9)/(4\sigma-1)(1-\sigma), & 0.91591\ldots \le \sigma \le 1-\eps.
\end{cases}
$$
\end{lemma}
\begin{proof}
See \cite[Th. 8.4]{ivic2003}.
\end{proof}

\section{Main results}

\begin{theorem}\label{th:lae-1}
$$
\sum_{n\le x} \lae(n)
= c_{11} x + c_{13} x^{1/3} + (c_{15}' \log x + c_{15}) x^{1/5}
+ O(x^{1153/6073+\eps}),
$$
where $c_{11}$, $c_{13}$, $c_{15}$ and $c_{15}'$ are computable constants.
\end{theorem}

\begin{proof}
Lemma~\ref{l:lae-series} and equation \eqref{eq:lae-series-1} implies that
$
\lae = \tau(1,3,5,5;\cdot) \star g_1
$,
where $\sum_{n\le x} g_1(n) \ll x^{1/6+\eps}$.
Due to \cite{kratzel1988}
\begin{multline*}
\sum_{n\le x} \tau(1,3,5,5; n)
= x \zeta(3) \zeta^2(5) \res_{s=1} \zeta(s)
+ 3 x^{1/3} \zeta(1/3) \zeta^2(5/3) \res_{s=1/3} \zeta(3s)
+
\\
+ 5 x^{1/5} \zeta(1/5) \zeta(3/5) \res_{s=1/5} \zeta^2(5s)
+ R(x).
\end{multline*}
To estimate $R(x)$ we use Lemma~\ref{l:kratzel} with $a_1=1$, $a_2=3$, $a_3=a_4=5$. Exponent pair
$
(k,l) = H_{2005}$ satisfies conditions 1, 2 and 3.2 and thus
$$
R(x) \ll x^{(k+l+2)/(k+l+14)} = x^{1153/6073+\eps},
\qquad
1/6 < 1153/6073 < 1/5.
$$
Now the convolution argument completes the proof.
\end{proof}

Exponential totient function $\phie$ has similar to $\lae$ Dirichlet series:
$$
\sum_{n=1}^\infty \phie(n) = \zeta(s) \zeta(3s) \zeta^2(5s) H(s),
$$
where $H(s)$ converges absolutely for $\sigma>1/6$. Theorem \ref{th:lae-1} can be extended to this case without any changes, so
$$
\sum_{n\le x} \phie(n)
= c_{11} x + c_{13} x^{1/3} + (c_{15}' \log x + c_{15}) x^{1/5}
+ O(x^{1153/6073+\eps}).
$$
This improves the result of Pétermann \cite{petermann2010}, who obtained
$
\sum_{n\le x} \phie(n)
\hm= c_{11} x \hm+ c_{13} x^{1/3} + O(x^{1/5} \log x)
$.

\medskip
{\bf Update from 23.05.2014:} Recently Cao and Zhai \cite[(1.13)]{caozhai2013} obtained $R(x) \hm\ll x^{18/95+\eps}$, which is better than both Pétermann's $x^{1/5} \log x$ and our $x^{1153/6073+\eps}$.

\begin{theorem}\label{th:lae-2}
$$
\sum_{n\le x} \lae_2(n)
= c_{22} x^{1/2} + (c_{23}'\log x + c_{23}) x^{1/3} + c_{24} x^{1/4}
+ O(x^{1153/5586+\eps}),
$$
where $c_{22}$, $c_{23}$, $c_{23}'$ and $c_{24}$ are computable constants.
\end{theorem}

\begin{proof}
Similar to Theorem~\ref{th:lae-1} with following changes: now by \eqref{eq:lae-series-2}
$$
\lae_2 = \tau(2,3,3,4; \cdot) \star g_2,
$$
where $\sum_{n\le x} g_2(n) \ll x^{1/6+\eps}$. But
\begin{multline*}
\sum_{n\le x} \tau(2,3,3,4; n)
= 2 x^{1/2} \zeta^2(3/2) \zeta(2) \res_{s=1/2} \zeta(2s)
+
\\
+ 3 x^{1/3} \zeta(2/3) \zeta(4/3) \res_{s=1/3} \zeta^2(3s)
+ 4 x^{1/4} \zeta(1/2) \zeta^2(3/4) \res_{s=1/4} \zeta(4s) + R(s).
\end{multline*}
Again by Lemma~\ref{l:kratzel} with $a_1=2$, $a_2=a_3=3$, $a_4=4$, $(k,l)=H_{2005}$ we get
$$
R(x) \ll x^{(k+l+2)/(k+l+12)} = x^{1153/5586+\eps},
\qquad
1/5 < 1153/5586 < 1/4.
$$
\end{proof}

\begin{theorem}\label{th:lae-3}
\begin{multline}\label{eq:lae-asymp-3}
\sum_{n\le x} \lae_3(n) =
(c_{33}' \log x + c_{33}) x^{1/3}
+ (c_{34}' \log x + c_{34}) x^{1/4}
+\\+ P_{35}(\log x) x^{1/5} + O(x^{1/6+\eps}),
\end{multline}
where $c_{33}$, $c_{33}'$, $c_{34}$ and $c_{34}'$ are computable constants,
$P_{35}$ is a polynomial of degree 3 with computable coefficients.
\end{theorem}

\begin{proof}
Lemma~\ref{l:lae-series} and equation \eqref{eq:lae-series-3} implies that
$
\lae_3 = z_3 \star g_3
$,
where $z_3$ is defined implicitly by
$$
\sum_{n=1}^\infty z_3(n) n^{-s} = Z_3(s) = \zeta^2(3s) \zeta^2(4s) \zeta^4(5s),
$$
and $g_3$ is a multiplicative function such that $\sum_{n\le x} g_3(n) \ll x^{1/6+\eps}$.

The main term at the right side of \eqref{eq:lae-asymp-3}
equals to
$$
M_3(x) := \left( \res_{s=1/3} + \res_{s=1/4} + \res_{s=1/5}  \right) \bigl( \zeta^2(3s) \zeta^2(4s) \zeta^4(5s) x^s s^{-1} \bigr).
$$
To obtain the desirable error term it is enough to prove that
$$
\sum_{n\le x} z_3(n) = M_3(x) + O(x^{1/6+\eps}).
$$
By Perron formula for $c:=1/3+1/\log x$ we have
$$
\sum_{n\le x} z_3(n)
= {1\over 2\pi i} \int_{c-iT}^{c+iT} Z_3(s) x^s s^{-1} ds
+ O(x^{1+\eps} T^{-1}).
$$
Substituting $T=x$ and moving the contour of the integration till $[1/6 \hm -ix, 1/6 \hm+ ix]$ we get
$$
\sum_{n\le x} f_3(n) = M_3(x) + O(I_0 + I_- + I_+ + x^\eps),
$$
where
$$
I_0 := \int_{1/6-ix}^{1/6+ix} Z_3(s) x^s s^{-1} ds,
\qquad
I_\pm := \int_{1/6\pm ix}^{c\pm ix} Z_3(s) x^s s^{-1} ds.
$$

Firstly,
$$
I_+ \ll x^{-1} \int_{1/6}^c Z_3(\sigma+ix) x^\sigma d\sigma.
$$
Let $\alpha(\sigma)$ be a function such that $Z_3(\sigma+ix) \ll x^{\alpha(\sigma)+\eps}$. By \eqref{eq:convexity} we have
$$
\alpha(\sigma) \le \begin{cases}
(16-68\sigma)\mu(1/2) < 4/5, & \sigma \in [1/6, 1/5), \\
\p(8-28\sigma)\mu(1/2) < 3/4, & \sigma \in [1/5, 1/4), \\
\p(4-12\sigma)\mu(1/2) < 2/3, & \sigma \in [1/4, 1/3), \\
\phantom{(0}0, & \sigma \in [1/3, c].
\end{cases}
$$
This means that $I_+ \ll x^\eps$. Plainly, the same estimate holds for $I_-$.

Secondly, it remains to prove that $I_0 \ll x^{1/6+\eps}$. Here
$$
I_0 \ll x^{1/6} \int_1^x Z_3(1/6+it) t^{-1} dt
$$
and taking into account Lemma~\ref{l:log-sum} it is enough to show $\int_1^x Z_3(1/6+it)  dt \hm\ll x^{1+\eps}$. Applying Cauchy inequality twice we obtain
\begin{multline*}
\int_1^x Z_3(1/6+it)  dt \ll
\left( \int_1^x \bigl|\zeta^4(1/2+it)\bigr| dt \right)^{1/2}
\times \\ \times
\left( \int_1^x \bigl|\zeta^8(2/3+it)\bigr| dt \right)^{1/4}
\left( \int_1^x \bigl|\zeta^{16}(5/6+it)\bigr| dt \right)^{1/4}
\ll \\
\ll x^{(1+\eps)\cdot 1/2} x^{(1+\eps)\cdot 1/4} x^{(1+\eps)\cdot 1/4}
\ll x^{1+\eps}
\end{multline*}
since by Lemma~\ref{l:ivic-moments-on-sigma} $m(1/2)\ge4$, $m(2/3)\ge 8$ and $m(5/6)\hm\ge 16$.
\end{proof}

\begin{theorem}\label{th:lae-4}
\begin{multline*}
\sum_{n\le x} \lae_4(n)
= (c_{44}' \log x + c_{44}) x^{1/4}
+ P_{45}(\log x) x^{1/5}
+ (c_{46}' \log x + c_{46}) x^{1/6}
+ \\
+ P_{47}(\log x) x^{1/7}
+ O(x^{C_4+\eps}),
\end{multline*}
where $c_{44}$, $c_{44}'$, $c_{46}$ and $c_{46}'$ are com\-put\-ab\-le constants,
$P_{45}$ and $P_{47}$ are com\-put\-ab\-le polynomials, $\deg P_{45} = 3$, $\deg P_{47} = 5$,
\begin{equation}\label{eq:C_4}
C_4 = {7863059 - \sqrt{13780693090921} \over 85962240} = 0.134656\ldots, \qquad 1/8 < C_4 < 1/7.
\end{equation}
\end{theorem}

\begin{proof}
We shall follow the outline of Theorem~\ref{th:lae-3}. Let us prove that for~$c:=1/4 \hm+ 1/\log x$ we can estimate
$$
I_+ := \int_{C_4+ix}^{c+ix} Z_4(s) x^s s^{-1} ds \ll x^{C_4+\eps}
$$
and
$$
I_0 := \int_{C_4-ix}^{C_4+ix} Z_4(s) x^s s^{-1} ds \ll x^{C_4+\eps}.
$$

We start with $I_+ \ll x^{-1} \int_{C_4}^c Z_4(\sigma+ix) x^\sigma d\sigma$. Now let $\alpha(\sigma)$ be a function such that $Z_4(\sigma+ix) \ll x^{\alpha(\sigma)+\eps}$. By \eqref{eq:convexity} and \eqref{eq:lae-series-4} we have
$$
\alpha(\sigma) \le \begin{cases}
(16-80\sigma)\mu(1/2) < 5/6, & \sigma \in [1/7, 1/6), \\
(12-56\sigma)\mu(1/2) < 4/5, & \sigma \in [1/6, 1/5), \\
\p(4-16\sigma)\mu(1/2) < 3/4, & \sigma \in [1/5, 1/4), \\
\phantom{(0}0, & \sigma \in [1/4, c].
\end{cases}
$$
So $\int_{1/7}^c Z_4(\sigma+ix) x^{\sigma-1} d\sigma \ll x^\eps$ and
the only case that requires further investigations is $\sigma\in[C_4, 1/7)$. Instead of \eqref{eq:convexity} we apply \eqref{eq:convexity-3/5}  together with \eqref{eq:mu-1/2} and \eqref{eq:mu-3/5} to obtain
$$
\alpha(\sigma) \le {1045018\over249485} - {2459357\over99794} \sigma, \qquad \sigma\in[1/8,1/7],
$$
which implies $\int_{C_4}^{1/7} x^{\alpha(\sigma) + \sigma - 1} d\sigma \ll x^{C_4+\eps}$ as soon as
$$
C_4 \ge 1591066/12296785 = 0.129388\ldots
$$
Our choice of $C_4$ in \eqref{eq:C_4} is certainly the case.

Let us move on $I_0$ and prove that $\int_1^x Z_4(C_4+it)\,dt \ll x^{1+\eps}$. For $q_1$, $q_2$, $q_3$, $q_4$ such that
\begin{equation}\label{eq:holder-conditions}
1/q_1+1/q_2+1/q_3+1/q_4=1 \qquad\text{and}\qquad q_1,q_2,q_3,q_4 \ge 1
\end{equation}
by Hölder inequality we have
\begin{multline*}
\int_1^x Z_4(C_4+it)\,dt
\ll
\left( \int_1^x \bigl| \zeta^{2q_1}(4s+it) \bigr| dt \right)^{1/q_1}
\left( \int_1^x \bigl| \zeta^{4q_2}(5s+it) \bigr| dt \right)^{1/q_2}
\sp
\times \\ \times
\left( \int_1^x \bigl| \zeta^{2q_3}(6s+it) \bigr| dt \right)^{1/q_3}
\left( \int_1^x \bigl| \zeta^{6q_4}(7s+it) \bigr| dt \right)^{1/q_4}
\sp
.
\end{multline*}
Choose
\begin{equation}\label{eq:choice-of-qk}
q_1 = m(4C_4)/2, ~~
q_2 = m(5C_4)/4, ~~
q_3 = m(6C_4)/2, ~~
q_4 = m(7C_4)/6
\end{equation}
One can make sure by substituting the value of $C_4$ from \eqref{eq:C_4} into Lemma~\ref{l:ivic-moments-on-sigma}  that such choice of $q_k$ satisfies \eqref{eq:holder-conditions}. Thus we obtain
$$
\int_1^x Z_4(C_4+it)\,dt \ll x^{(1+\eps)/q1} x^{(1+\eps)/q2} x^{(1+\eps)/q3} x^{(1+\eps)/q4} \ll x^{1+\eps},
$$
which finishes the proof.
\end{proof}

\section{Decrease of \texorpdfstring{$C_4$}{C4}}

In this section we obtain lower value of $C_4$ by improving lower bounds of $m(\sigma)$ from Lemma~\ref{l:ivic-moments-on-sigma}.

Estimates below depend on values of
\begin{equation}\label{eq:linear-programming}
\inf_{(k,l)} {ak+bl+c \over dk+el+f},
\end{equation}
where $(k,l)$ runs over the set of exponent pairs and satisfies certain linear inequalities. A method to estimate \eqref{eq:linear-programming} without linear constrains was given by Graham~\cite{graham1986}. In the recent paper \cite{lelechenko2013-acta} we have presented an effective algorithm to deal with \eqref{eq:linear-programming} under a  nonempty set of linear constrains.

\medskip

Let $c$ be an arbitrary function such that $c(\sigma) \ge \mu(\sigma)$. Define $\theta$ by an implicit equation
$$
2 c\bigl(\theta(\sigma)\bigr) + 1 + \theta(\sigma) - 2 \bigl( 1+c\bigl(\theta(\sigma)\bigr) \bigr) \sigma = 0.
$$
Finally, define
$$
f(\sigma) = 2{1+c\bigl(\theta(\sigma)\bigr) \over c\bigl(\theta(\sigma)\bigr)}.
$$
Due to Lemma~\ref{l:order-in-the-critical-strip} one can take $c(\sigma) = \inf_{l-k\ge\sigma} (k+l-\sigma)/2$, where $(k,l)$ runs over the set of exponent pairs. However even rougher choice of $c$ leads to satisfiable values of $f$ such as in \cite[(8.71)]{ivic2003}.

\begin{lemma}\label{l:pointwise-moments-on-sigma}
Let $\sigma\ge5/8$. Compute
$$
\alpha_1 = {4-4\sigma \over 1+2\sigma},
\qquad
\beta_1 = - {12\over 1+2\sigma},
\qquad
m_1 = {1-\alpha_1 \over \mu(\sigma)} - \beta_1,
$$
$$
\alpha_2(k,l)
= {4(1-\sigma)(k+l) \over (2+4l)\sigma-1+2k-2l},
\qquad
\beta_2(k,l)
= -{4(1+2k+2l) \over (2+4l)\sigma-1+2k-2l},
$$
$$
m_2(k,l) = {1-\alpha_2(k,l) \over \mu(\sigma)} - \beta_2(k,l),
\qquad
m_2 = \inf_{\alpha_2(k,l) \le 1  } m_2(k,l),
$$
where $(k,l)$ runs over the set of exponent pairs. Then
$$
m(\sigma) \ge \min\bigl(m_1, m_2, 2f(\sigma)\bigr).
$$
Note that for $\sigma\ge2/3$ the condition $\alpha_2(k,l) \le 1$ is always satisfied.
\end{lemma}

\begin{proof}
Follows from \cite[(8.97)]{ivic2003} and from
$
T^\alpha V^\beta \ll T V^{\beta + (\alpha-1) / \mu(\sigma)}
$
for $\alpha<1$ and~$V \ll T^{\mu(\sigma)}$.
\end{proof}


Substituting pointwise estimates of $m(\sigma)$ from Lemma~\ref{l:pointwise-moments-on-sigma} instead of seg\-ment\-wi\-se from Lemma~\ref{l:ivic-moments-on-sigma} into \eqref{eq:choice-of-qk} we obtain following result.

\begin{theorem}
The statement of Theorem~\ref{th:lae-4} remains valid for
$$
C_4 = 0.133437785\ldots
$$
\end{theorem}

\section{Conclusion}

We have obtained nontrivial error terms in asymptotic estimates of $\sum_{n\le x} \lae_r(n)$ for $r=1,2,3,4$. Cases of $r=1$ and $r=2$ depend on the method of exponent pairs. Cases of $r=3$ and $r=4$ depend on lower bounds of $m(\sigma)$. Note that case of $r=4$ may be improved under Riemann hypothesis up to $C_4=1/8$, because Riemann hypothesis implies $\mu(\sigma)=0$ and $m(\sigma)=\infty$ for $\sigma\in[1/2,1]$.

\bibliographystyle{ugost2008s}
\bibliography{taue}

\end{document}